\documentclass[reqno]{amsart}
\usepackage{}
\usepackage[top=1in, bottom=1in, left=1in, right=1in]{geometry}
\usepackage{cases}
\usepackage[colorlinks,
            linkcolor=red,
            anchorcolor=blue,
            citecolor=green
            ]{hyperref}
\usepackage{amsmath,amsthm,amsfonts,amssymb,amscd,bm}
\usepackage{url}
\newtheorem{theorem}{Theorem}[section]
\newtheorem{lemma}[theorem]{Lemma}
\theoremstyle{definition}
\newtheorem{definition}[theorem]{Definition}
\newtheorem{example}[theorem]{Example}

\newtheorem{prop}[theorem]{Proposition}

\newtheorem{cor}[theorem]{Corollary}

\theoremstyle{remark}
\newtheorem{remark}[theorem]{Remark}
\renewcommand{\dim}{\mbox{\rm dim}}

\newcommand{\GL}{\mbox{\rm GL}}

\newcommand{\SL}{\mbox{\rm SL}}

\newcommand{\Gal}{\mbox{\rm Gal}}
\newcommand{\MOD}{\mbox{\rm mod~}}
\newcommand{\GQ}{\Gal(\overline{\mathbb{Q}}/\mathbb{Q})}
\newcommand{\Frob}{\mbox{\rm Frob}}
\newcommand{\ord}{\mbox{\rm ord}}
\numberwithin{equation}{section}

\begin{document}\large

\title{On the computation of coefficients of modular forms: the reduction modulo p approach}

\author{Jinxiang Zeng and Linsheng Yin}
\address{Department of Mathematical Science, Tsinghua University, Beijing 100084, P. R. China}
\email{cengjx09@mails.tsinghua.edu.cn}


\subjclass[2012]{Primary 11F30, 11G20, 11Y16, 14Q05, 14H05}



\keywords{modular forms, Hecke algebra, modular curves, elliptic curves, Jacobian}

\begin{abstract}
In this paper, we present a probabilistic algorithm to compute the coefficients of modular forms of level one. Focusing on the Ramanujan's tau function, we give the explicit complexity of the algorithm. From
a practical viewpoint, the algorithm is particularly well suited for implementations.
\end{abstract}

\maketitle

\section{Introduction and Main Results}
In the book \cite{Edixhoven}, Couveignes, Edixhoven et el. described an algorithm for computing
coefficients of modular forms for the group SL$_2(\mathbb{Z})$, and Bruin \cite{Bruin} generalized
the method to modular forms for the congruence subgroups of the form $\Gamma_1(n)$.
Their methods lead to  polynomial time algorithms for computing coefficients of modular forms.
However, efficient ways to implement the algorithms and explicit complexity analysis are still being
studied. Working with complex number field, Bosman's explicit computations show the power of these new
methods, see \cite{Edixhoven}. As one of the applications, he largely improved the known result on
Lehmer's nonvanishing conjecture for Ramanujan's tau function.
For the recent progress in this direction see \cite{Mascot}.
Following Couveignes's idea \cite{Couveignes}, we give  a probabilistic algorithm,
which seems to be more suitable to deal with complexity analysis.
Instead of using Brill-Noether's algorithm, we work with the function field of the modular curve,
using He\ss's algorithm to make computations in the Jacobian of the modular curve.

We illustrate our method on the discriminant modular form, which is defined as
$$\Delta(q)=q\prod_{n=1}^{\infty}(1-q^n)^{24}=\sum_{n=1}^{\infty}\tau(n)q^n,$$
where $z \in \mathcal{H},q=e^{2\pi i z}$.

Let $p$ be a prime, using Deligne's bound we have $|\tau(p)|\le2p^{11/2}$, therefore to compute $\tau(p)$
it suffices to compute $\tau(p)\mod\ell$ for all primes $\ell$ bounded above by a constant in $\textrm{O}(\log p)$.
Let $\ell$ be a prime, the mod-$\ell$ Galois representation associated to $\Delta(q)$ is denoted as
$$\rho_\ell:\GQ\to \GL_2(\mathbb{F}_\ell),$$
which satisfying that for any prime $p\not=\ell$, the characteristic polynomial of the Frobenius
endomorphism $\textrm{Frob}_p$ is $x^2-\tau(p)x+p^{11}\mod \ell$. Therefore,
to compute $\tau(p)\mod\ell$ it suffices to compute the Galois representation $\rho_\ell$.
It is well known that $\rho_\ell$ can be realized by the subspace $V_\ell$ in the $\ell$-torsions
 $J_1(\ell)(\overline{\mathbb{Q}})[\ell]$ of the Jacobian variety $J_1(\ell)$ of the
 modular curve $X_1(\ell)$, which can be written as a finite intersection
$$V_\ell=\bigcap_{1\le i\le (\ell^2-1)/6} \ker(T_i-\tau(i),J_1(\ell)(\overline{\mathbb{Q}})[\ell]),$$
where $T_i,i=1 ,\ldots,\frac{\ell^2-1}{6}$ are Hecke operators.
Indeed $V_\ell$ is a group scheme over $\mathbb{Q}$ of order $\ell^2$,
which is called the Ramanujan subspace \cite{Couveignes}. As showed in \cite{Edixhoven},
the heights of the elements of $V_\ell$ are well bounded, which enables us to know $V_\ell$ explicitly.
More precisely, we have function $\iota:V_\ell\to \mathbb{A}_{\mathbb{Q}}^1$,
such that the heights of the coefficients of
$P(X):=\prod_{\alpha\in V_\ell\setminus\textrm{O}}(X-\iota(\alpha))$ are bounded above by
$\textrm{O}(\ell^\delta)$, where the constant $\delta$ is independent of $\ell$ and the function
$\iota$ is constructed explicitly in \cite{Bruin}.

Our approach is to compute $V_\ell\mod p$  for sufficiently many auxiliary small primes $p$ as
in Schoof's algorithm, and then reconstruct $V_\ell$ by the Chinese Remainder Theorem.
The main results of the paper are as follows.

\begin{theorem}\label{theorem:complexity}Let $\ell\ge 13$ be a prime and $p$ an $s$-good prime.
Given the Zeta function of the modular curve $X_1(\ell)_{\mathbb{F}_p}$, then $V_\ell \mod p$
can be computed in time $\textrm{O}(\ell^{4+2\omega+\epsilon}\log^{1+\epsilon}p\cdot(\ell+\log p))$.
\end{theorem}

\begin{cor}\label{cor:complexity}(1)The Ramanujan subspace $V_\ell$ can be computed in time $\textrm{O}(\ell^{5+2\omega+\delta+\epsilon})$.

(2) For prime $p$, $\tau(p)$ can be computed in time $\textrm{O}(\log^{6+2\omega+\delta+\epsilon} p)$.
\end{cor}
\begin{remark}See \ref{s-goodprime} for the definition of $s$-good prime. The constant $\omega$ refers to that,
the complexity of a single group operation in the Jacobian variety $J_1(\ell)$ is in $\textrm{O}(g^\omega)$,
where $g$ is the genus of the modular curve $X_1(\ell)$. Using Khuri-Makdisi's algorithm, the constant $\omega$ can be $2.376$.
Using He\ss's algorithm, $\omega$ is known to be in $[2,4]$. The constant $\delta$ is bigger than $\dim V_\ell=2$.
\end{remark}

\begin{theorem}\label{theorem:main2} The nonvanishing of $\tau(n)$ holds for all $n$ such that
$$n<982149821766199295999\approx9\cdot10^{20}.$$
\end{theorem}
\begin{remark}In \cite{Bosman} the nonvanishing of $\tau(n)$ was verified for all $n$ such that
$$n < 22798241520242687999 \approx 2 \cdot10^{19}.$$
\end{remark}

{\bf Notation}: The running time will always be measured in bit operations. Using FFT, multiplication of two $n$-bit
length integers can be done in $\textrm{O}(n^{1+\epsilon})$ time. Multiplication in finite field $\mathbb{F}_q$ can be done
in $\textrm{O}(\log^{1+\epsilon} q)$.

The paper is organized as follows.
In Section 2 we provide some necessary background on computing a convenient plane model and the function field of the modular curve,
results for computing isogenies of elliptic curve over finite fields are also recalled. A better bound on the generators of the maximal
ideal of Hecke algebra is proved, which is used to reduce the complexity of the algorithm.

Section 3 contains the application of He\ss's algorithm to the computation in the Jacobian of the modular curve over finite fields.
Here, we introduce methods to find the correspondence between the places of the function field and the cusps of the modular curve,
and compute the action of Hecke operators on places of the function field.

The main algorithm is given in Section 4, including complexity analysis.

Section 5 is concerned with some real computations of the Ramanujan's tau function.
All of our computations are based on Magma computational algebra system \cite{Bosma}.

\section{Function field of modular curves}
In this section we study the plane model and the function field of the modular curve $X_1(\ell)$.

Let $\Gamma_1(\ell)$ be a congruence subgroup of SL$_2(\mathbb{Z})$, defined as
$$\Gamma_1(\ell)=\left\{ \left[\begin{matrix} a & b\\ c & d \end{matrix}\right]\in \textrm{SL}_2(\mathbb{Z}):
\left[\begin{matrix} a & b\\ c & d \end{matrix}\right]\equiv
\left[\begin{matrix} 1 & *\\ 0 & 1 \end{matrix}\right](\MOD \ell)  \right\},$$
(where `` * '' means `` unspecified '') and
 $\mathcal{H}$ the upper half complex plane. The modular curve $Y_1(\ell)$ is defined as the quotient
space of orbits under $\Gamma_1(\ell)$,
$$Y_1(\ell)=\Gamma_1(\ell)\backslash \mathcal{H}.$$
We can add cusps $\mathbb{P}^1(\mathbb{Q})$ to $Y_1(\ell)$ to compactify it and obtain the modular curve
$$X_1(\ell)=\Gamma_1(\ell)\backslash \mathcal{H}\cup \mathbb{P}^1(\mathbb{Q}).$$
This complex algebraic curve is defined over $\mathbb{Q}$, denoted by $X_1(\ell)_{\mathbb{Q}}$.
Moreover, for $\ell\ge 5$, $X_1(\ell)$ has natural model over $\mathbb{Z}[1/\ell]$.
Let $K$ be a number field, then $K$-valued points of $Y_1(\ell)_{\mathbb{Q}}$ can be interpreted as
$$Y_1(\ell)_{\mathbb{Q}}(K)=\{(E,P):E/K, P\in E[\ell](K)\setminus O\}/{\sim},$$
where $E$ is an elliptic curve over $K$, $P$ is a $K$-rational point of order $\ell$, and $(E_1,P_1)\sim(E_2,P_2)$ means that,
there exists a $\overline{K}$-isomorphism $\phi:E_1\to E_2$, such that $\phi(P_1)=P_2$. Such a moduli interpretation implies a
way to obtain a plane model of the modular curve, as the following proposition (see \cite{BAAZIZ}),

\begin{prop}Suppose that $\ell \ge 4$. Then every $K$-isomorphism class of pairs $(E,P)$ with $E$ an elliptic curve over $K$ and
$P \in E(K)$ a torsion point of order $\ell$
contains a unique model of the Tate normal form
\begin{equation}\label{TateNormalForm}
E_{(b,c)}:y^2+(1-c)xy-by =x^3-bx^2,P=(0,0),
\end{equation}
with $c\in  K,b \in K^*$.
\end{prop}
Thus, points of $X_1(\ell)$ can be represented as pairs $(b,c)$ in a unique way. The $\ell$-th division polynomial gives a polynomial
in $b$ and $c$, which defines a plane curve birationally equivalent to $X_1(\ell)$. The defining equation becomes much simpler,
through a carefully chosen sequence
of rational transformations. We use the table listed in \cite{Sutherland12}, for example a plane model of $X_1(19)$ is
\begin{equation}
\begin{split}
f(x,y)=&y^5 - (x^2 + 2)y^4 - (2x^3 + 2x^2 + 2x - 1)y^3+ (x^5 + 3x^4 + 7x^3 + 6x^2 + 2x)y^2\\
&-(x^5 + 2x^4 + 4x^3 + 3x^2)y + x^3 + x^2,
\end{split}
\end{equation}
where
\begin{equation}\label{equation:r(xy)}
r=1+\frac{x(x+y)(y-1)}{(x+1)(x^2-xy+2x-y^2+y)},s=1+\frac{x(y-1)}{(x+1)(x-y+1)},
\end{equation}
and
\begin{equation}
c=s(r-1),b=rc.
\end{equation}
So the function field of $X_1(19)$ over $\mathbb{Q}$ is
$$\mathbb{Q}(X_1(19))=\mathbb{Q}(b,c)=\mathbb{Q}(x)[y]/(f(x,y)).$$

Cusps of $X_1(\ell)$ correspond to those pairs $(b,c)$ such that the $j$-invariant $j(E_{(b,c)})=\infty$.
For $\ell$ an odd prime, the modular curve $X_1(\ell)$ has $\ell-1$ cusps,
half of which are in $X_1(\ell)(\mathbb{Q})$ and the rest are defined over
the maximal real subfield of $\mathbb{Q}(\zeta_\ell)$.
Accordingly, every $\mathbb{Q}$-rational cusp corresponds to a degree one
place of $\mathbb{Q}(X_1(\ell))$, denote the rational cusps as $O_1,\ldots,O_{(\ell-1)/2}$,
the $\mathbb{Q}(\zeta_\ell)$-cusps corresponds to a degree $\frac{\ell-1}{2}$
place of $\mathbb{Q}(X_1(\ell))$. It is easy to get these places after writing down the exact
expression of $j(E_{(b,c)})$ in variables $x,y$. As in the above example,
one of the $\mathbb{Q}$-rational cusps looks like
\begin{displaymath}
\begin{split}
O_1=&\left(x,\frac{y^4}{x^4+x^3}+\frac{y^3(-x^3-x^2+x-1)}{x^4+x^3}+ \frac{y^2(-x^3-2x^2-2x-2)}{x^3+x^2}+\frac{2y}{x}+\frac{2x-1}{x}\right),
\end{split}
\end{displaymath}
where the place $O_1$ is represented by a prime ideal of the maximal orders of the function field $\mathbb{Q}(X_1(\ell))$.

It's well known that, the modular curve $X_1(\ell)$ has good reduction at prime $p\nmid \ell$, see \cite{Diamond}. The reduction curve is
denoted by $X_1(\ell)_{\mathbb{F}_p}$. Having a nonsingular affine model of $X_1(\ell)$, we can easily get an affine model
for $X_1(\ell)_{\mathbb{F}_p}$ and  then have the function field of $X_1(\ell)_{\mathbb{F}_p}$. For simplicity, the plane model
of $X_1(\ell)_{\mathbb{F}_p}$ and the $\mathbb{F}_p$-rational cusps of $X_1(\ell)_{\mathbb{F}_p}$, which are the reductions of
the $\mathbb{Q}$-rational cusps of $X_1(\ell)$, are also denoted by $f(x,y)$ and $O_i,i \in\{1,\ldots,\frac{\ell-1}{2}\}$, respectively.

The Ramanujan subspace $V_\ell \mod p$ is a subgroup scheme of the Jacobian variety  $J_1(\ell)_{\mathbb{F}_p}$  of $X_1(\ell)_{\mathbb{F}_p}$.
Similarly, it can be written as a finite intersection
$$V_\ell \mod p=\bigcap_{1\le i \le \frac{\ell^2-1}{6}}\ker(T_i-\tau(i),J_1(\ell)_{\mathbb{F}_p}[\ell]),$$
where $T_i,~1\le i\le \frac{\ell^2-1}{6}$ are Hecke operators, the number $\frac{\ell^2-1}{6}$ follows from \cite{Stu87}.
In fact, the Hecke algebra $\mathbb{T}=\mathbb{Z}[T_n:n\in \mathbb{Z}^+]$ $\subset$ End$(J_1(\ell))$ is a free $\mathbb{Z}$-module
of rank $g=\frac{(\ell-5)(\ell-7)}{24}$. After representing each Hecke operator as a matrix, see \cite{Edixhoven}, the generators can
be extracted from $T_1,\ldots,T_{(\ell^2-1)/6}$ by solving linear equations. For example, when the level $\ell=17$, the Hecke
algebra $\mathbb{T}$ is equal to $\mathbb{Z}T_1+\ldots+\mathbb{Z}T_{48}$ as an $\mathbb{Z}$-module, which can be replaced
by $\mathbb{Z}T_1+\mathbb{Z}T_2+\mathbb{Z}T_3+\mathbb{Z}T_4+\mathbb{Z}T_6$ as a free $\mathbb{Z}$-module,
so there are fewer Hecke operators and isogenies of lower degrees needed to take into account. But, in practice,
we can do much better, notice that our goal is to find nonzero elements in $J_1(\ell)_{\mathbb{F}_p}[\ell]$,
which are canceled by $T_k-\tau(k), k\ge 1$. Assume there is an element $D\in J_1(\ell)_{\mathbb{F}_p}[\ell]$
satisfying $(T_2-\tau(2))(D)=0$, then if we have the relations  $T_k-\tau(k)=\phi_k\cdot(T_2-\tau(2))$,
for some endomorphism $\phi_k\in$ End($J_1(\ell)_{\mathbb{F}_p}$) in advance, then $(T_k-\tau(k))(D)$
is equal to zero automatically, which implies that $D\in V_\ell$. The action of $T_k-\tau(k)$
on $J_1(\ell)_{\mathbb{F}_p}[\ell]$ can be represented by a matrix over finite field $\mathbb{F}_\ell$,
and the existence of $\phi_k$ is equivalent to the existence of some matrix $M_k$ over $\mathbb{F}_\ell$
such that $T_k-\tau(k)=M_k\cdot(T_2-\tau(2))$. For example, when the level $\ell\in \{13,17,19,23,29,37,41,43\}$,
for any $k\ge 3$, we have $T_k-\tau(k)=M_k\cdot( T_2-\tau(2))$, for some matrix $M_k$ over finite field $\mathbb{F}_\ell$,
while $\ell=31$, $T_3-\tau(3)$ also satisfies this property. In fact, we proved the following proposition.

\begin{prop}\label{propositon:linearlyinl}Let $\ell$ be a prime bigger than 3 and $S_2(\Gamma_1(\ell))$
the space of cusp modular forms of weight 2 level $\ell$ over $\mathbb{C}$.
Let $\mathbb{T}=\mathbb{Z}[T_n:n\ge1]$ $\subset$ $\textrm{End}(S_2(\Gamma_1(\ell)))$ be the Hecke algebra
and $\mathfrak{m}$ the maximal ideal generated by $\ell$ and the $T_n-\tau(n)$ with $n\ge1$.
Then $\mathfrak{m}$ can be generated by $\ell$ and the $T_n-\tau(n)$ with $1\le n\le \lceil\frac{2\ell+1}{12}\rceil$.
\end{prop}
\begin{proof}Let $b_\ell:=\lceil\frac{2\ell+1}{12}\rceil$ and $R:=\mathbb{T}\otimes_{\mathbb{Z}}\mathbb{F}_\ell $,
then $R$ is an Artin ring, which can be decomposed as
$$R=\prod_{\wp}R_{\wp},$$
where $\wp$ runs through all maximal ideals of $R$, and $R_{\wp}$ is the localization of $R$ at $\wp$.

Let $\tilde{\mathfrak{m}}$ be the image of $\mathfrak{m}$ in $R$. Then, it is enough to prove $\tilde{\mathfrak{m}}$
can be generated by the $T_n-\tau(n)$ with $n\le b_\ell$, which is equivalent to show that any $T_k-\tau(k),k>b_\ell$ can be represented as
\begin{equation}\label{equation:prop1}
T_k-\tau(k)=\sum_{i=1}^{b_\ell}A_i\cdot(T_i-\tau(i)),
\end{equation}
where $A_i$ are operators in $R$.

Let $S_2(\Gamma_1(\ell);\overline{\mathbb{F}}_\ell):=S_2(\Gamma_1(\ell);\mathbb{Z})\otimes \overline{\mathbb{F}}_\ell$,
which is an $R$-module, decomposed as
$$S_2(\Gamma_1(\ell);\overline{\mathbb{F}}_\ell)=\prod_{\wp}S_2(\Gamma_1(\ell);\overline{\mathbb{F}}_\ell)_{\wp},$$
where $\wp$ runs through all maximal ideals of $R$, and $S_2(\Gamma_1(\ell);\overline{\mathbb{F}}_\ell)_{\wp}$ is the
localization of $S_2(\Gamma_1(\ell);\overline{\mathbb{F}}_\ell)$ at $\wp$, which is an $R_{\wp}$-module.
To show (\ref{equation:prop1}), it's enough to show for each $\wp$, the action $T_k-\tau(k),k>b_\ell$
on $S_2(\Gamma_1(\ell);\overline{\mathbb{F}}_\ell)_{\wp}$ can be represented as
\begin{equation}\label{equation:prop2}
T_k-\tau(k)=\sum_{i=1}^{b_\ell}B_i\cdot(T_i-\tau(i)),
\end{equation}
where $B_i$ are operators in $R_{\wp}$.

The maximal ideal $\wp$ of $R$ corresponds to a $\textrm{Gal}(\overline{\mathbb{F}}_\ell/\mathbb{F}_\ell)$-conjugate
class $[f]$ of normalized eigenforms in $S_2(\Gamma_1(\ell);\overline{\mathbb{F}}_\ell)$,
they are newforms of level $\ell$. Thus the localization $S_2(\Gamma_1(\ell);\overline{\mathbb{F}}_\ell)_{\wp}$ is a
vector space spanned by these newforms, so $R_{\wp}$ is a  field isomorphic to the field $\mathbb{F}_{\ell}(f)$,
which is generated by the coefficients of $f$ over $\mathbb{F}_\ell$.

Let $a_i(f)$ be the $i$-th coefficient of the Fourier expansion of $f$,
to show (\ref{equation:prop2}) it is enough to show $a_k(f)-\tau(k)$ is always equal to zero,
or there exists at least one nonzero element among the $a_i(f)-\tau(i)$ with $1\le i\le b_\ell$.

If $f$ is congruent to $\Delta(q)\mod\ell$, then $a_k(f)-\tau(k)$ is always equal to zero.
Now suppose $f$ is not congruent to $\Delta(q)$ modulo $\ell$ and $a_i(f)-\tau(i)=0$ for
$1\le i\le b_\ell$. The Proposition 4.10(b) of \cite{Gross} together with  Theorem 3.5(a)
of \cite{Ash-Ste} imply that $f$ comes from a level one newform of weight $k_1$ with $k_1\le2\ell$.
Since for any prime $p$, $\tau(p^2)=\tau(p)^2-p^{11}$ and $a_{p^2}(f)=a_p(f)^2-p^{k_1-1}$,
we have $p^{k_1-12}\equiv 1\mod \ell$ for all primes $p$ satisfying $p^2\le b_\ell$.
Since the least primitive root modulo $\ell$ is in $\textrm{O}(\ell^{\frac{1}{4}+\epsilon})$,
we have $k_1-12\equiv 0\mod(\ell-1)$.
Let $A(q)=1$ be the Hass invariant, then $f_1:=A(q)^{(k_1-12)/(\ell-1)}\Delta(q)$ is a newform of
level one weight $k_1$, and $a_k(f_1)=a_k(f)$ for all $k\le b_\ell$. Since $b_\ell=\lceil\frac{2\ell+1}{12}\rceil\ge\frac{k_1+1}{12}$,
the theorem of Sturm \cite{Stu87} (or see \cite{CG11}) tells that $f_1$ is congruent to $f$, that is $\Delta(q)$ is congruent to $f$
modulo $\ell$,
which is a contradiction. So the proposition is proved..
\end{proof}

Since $\ell$ and a subset of $\{T_k-\tau(k):1\le k\le \lceil\frac{2\ell+1}{12}\rceil\}$ may also suffice to
generate $\mathfrak{m}$, we introduce the optimal subset as follows.
\begin{definition}\label{definition:optimalset}Let $\mathcal{S}$ be a set of positive integers,
such that $\ell$ and $T_n-\tau(n),n\in\mathcal{S}$ generate $\mathfrak{m}$. The set $\mathcal{S}$ is
called optimal if $\prod_{n\in \mathcal{S}}n$ is minimal among all the $\mathcal{S}$.
\end{definition}
In practice only those operators in the optimal subset need to be considered, which will accelerate the algorithm.

Let $n\ge 2$ be a prime not equal to $\ell$, and $Q$ a point of $X_1(\ell)_{\mathbb{F}_p}(\overline{\mathbb{F}}_p)$
represented by $(E_{(b,c)},(0,0))$, the computation of  $T_n(Q)$ comes down to computing isogenies of elliptic curves
over some finite extension fields of $\mathbb{F}_p$
$$T_n(E_{(b,c)},(0,0))=\sum_{C}(E/C,(0,0)+C),$$
where $C$ runs over all the order $n$ subgroups of $E_{(b,c)}$. The following result about the complexity of computing $n$-isogeny
is in \cite{BOSTAN}, when $n$ is small compared to the characteristic of the field,

\begin{prop}Let $\mathbb{F}_q$ be a finite field of characteristic $p$, $n$ a prime not equal to $p$, $E$ and $\tilde{E}$ two
elliptic curves over $\mathbb{F}_q$ in Weierstrass form. Assume there is a normalized isogeny $\phi:E\to \tilde{E}$ of
degree $n$, then $\phi$ can be computed in $O(n^{1+\epsilon})$ multiplications in the field $\mathbb{F}_q$.
\end{prop}
We give two notices here: The first is that, the original elliptic curve $E_{(b,c)}$ is defined over some finite field $\mathbb{F}$,
but the isogeny may lie in the extension field of $\mathbb{F}$. The exact definition field of the isogeny can be obtained by computing
the $n$-th classical modular polynomial $\Phi_n(X,Y)$, and solving the equation $\Phi_n(X,j(E_{(b,c)}))=0$, so the extension degree is
less or equal to the degree of $\Phi_n(X,j(E_{(b,c)}))$. The second is that, in general the isogenous curve $\tilde{E}$ is not in Tate
normal form (\ref{TateNormalForm}).  Using the map $\phi$, or rather the point $\phi((0,0))\in \tilde{E}$, $\tilde{E}$ can be transformed
into Tate normal form after some coordinate changes, this gives a new point on the modular curve $X_1(\ell)_{\mathbb{F}_p}$.

\section{Computing in the jacobian of modular curves}
Let $\mathbb{F}_q$ be a finite extension of the finite field $\mathbb{F}_p$ and $X_1(\ell)_{\mathbb{F}_q}$ the base change
of $X_1(\ell)_{\mathbb{F}_p}$ to $\mathbb{F}_q$. One of the most important tasks of our algorithm is computing in the
Jacobian $J_1(\ell)_{\mathbb{F}_q}$ of the modular curve $X_1(\ell)_{\mathbb{F}_q}$. For general curves, we already
have polynomial time algorithms to perform operation (addition and subtraction) in their Jacobians \cite{Huang} \cite{Volcheck} \cite{Hess} \cite{KM}.
For the Jacobian of modular curve, Couveignes uses Brill-Noether algorithm to do the computation \cite{Couveignes}, while Bruin uses Khuri-Makdisi's
algorithm \cite{Bruin}. We choose He\ss's algorithm, with the advantage that it's easy know the correspondences between points of the modular curve
and places of its function field.

Let's recall the main idea of He\ss's algorithm, for the detail see \cite{Hess}. Let $\mathbb{K}=\mathbb{F}_q(x)[y]/(f(x,y))$ be the function field
of the modular curve $X_1(\ell)_{\mathbb{F}_q}$. There are isomorphisms
$$J_1(\ell)(\mathbb{F}_q)\cong  \textrm{Pic}^0(X_1(\ell)_{\mathbb{F}_q})\cong \textrm{Cl}^0(\mathbb{K}).$$
Notice that, there are some calculations behind the second isomorphism, i.e. computing the change of representations. He\ss's algorithm is based on
the arithmetic of the function field $\mathbb{K}$. Let $\mathbb{P}$ be the set of all places of $\mathbb{K}$ and $S$ the set of the places
of $\mathbb{K}$ over the infinite place $\infty$ of $\mathbb{F}_q(x)$. The ring of elements of $\mathbb{K}$ being integral at all places
of $S$ and $\mathbb{P}\setminus S$ are denoted by $\mathcal{O}_S$ and $\mathcal{O}^S$, respectively. They are Dedekind domains,
called infinite and finite order of $\mathbb{K}$, whose divisor groups are denoted by $\textrm{Div}(\mathcal{O}_S)$ and  $\textrm{Div}(\mathcal{O}^S)$,
respectively. Place of $\mathbb{K}$ corresponds to prime ideal of $\mathcal{O}_S$ or $\mathcal{O}^S$. In fact,
let $\textrm{Div}(\mathbb{K})$ be the divisor group of $\mathbb{K}$, then $\textrm{Div}(\mathbb{K})$ can be decomposed as \cite{Hess}
$$\textrm{Div}(\mathbb{K})\xrightarrow{\sim}\textrm{Div}(\mathcal{O}_S)\times \textrm{Div}(\mathcal{O}^S).$$

Notice that the plane model of $X_1(\ell)$ given in \cite{Sutherland12} has singularities above $x=0$ and $x=-1$, we can check that places
of $\mathbb{K}$ over the places $(x)$,$(x+1)$ and $(\frac{1}{x})$ of $\mathbb{F}_q(x)$ are cusps of $X_1(\ell)$. We will discuss how to compute the
action of Hecke operators on these places in Section 3.2.

Now, let's focus on how to compute the action of Hecke operators on the place corresponding to a set of smooth points of $f(x,y)=0$. Such a
place $\wp$ can be represented by two elements of $\mathbb{K}^{\times}$, which can be normalized as $f_1(x)=x^d+a_{d-1}x^{d-1}+\ldots+a_1x+1$
and $f_2(x,y)=y^m+b_{m-1}(x)y^{m-1}+\ldots+b_1(x)y+b_0(x)$, where $a_i,0\le i\le d-1$, belong to the constant field $\mathbb{F}_q$,
and $b_i(x),~0\le i\le m-1$ are elements of $\mathbb{F}_q[x]$, with degrees less than $d$. So the point set corresponding to  $\wp$
can be computed as follows. Let $f_1(x)$ and $f_2(x,y)$ be the normalized generators of $\wp$. We first compute the roots of $f_1(x)=0$,
denoted by $x_i,1\le i\le d$. Then, for each root $x_i$, compute the roots of $f_2(x_i,y)=0$, denoted by $y_{ij},1 \le j\le m$.
The point set corresponding to $\wp$ is $\{(x_i,y_{ij}):{1\le i\le d,1\le j\le m} \}$. Remind that every point $(x_i,y_{ij})$
satisfies $f(x_i,y_{ij})=0$. Conversely, given a point set $\{(x_i,y_{ij}):{1\le i\le d,1\le j\le m} \}$, the two generators for
the corresponding prime ideal  $\wp$ of $\mathcal{O}^S$ can be computed as follows. The first generator is clear,
which is $f_1(x)=\prod_{i=1}^d(x-x_i)$. The second one can be recovered as follow: Let $b_i(x)=\sum_{k=0}^{d-1}c_{ik}x^k,0\le i\le m-1$,
where $c_{ik}$ are parameters belong to $\mathbb{F}_q$, after interpolating the points $(x_i,y_{ij})$, we have linear equations of $md$ variables,
the second generator can be known by solving these equations.

So, a degree $d$  place  $\wp$  as in above  corresponds to a point set, denoted by $\{(x_i,y_i):1\le i\le d\}$, which forms a
complete $\Gal(\mathbb{F}_{q^d}/\mathbb{F}_q)$-conjugate set. Using the coordinate transformation formulae (\ref{equation:r(xy)}),
for each point $(x_i,y_i)$, the corresponding point on $X_1(\ell)_{\mathbb{F}_p}$ of the form $(E_{(b_i,c_i)},(0,0))$ is clear,
where $E_{(b_i,c_i)}$ is an elliptic curve in Tate normal form, and $(0,0)$ is a point of order $\ell$. As discussed in Section 2,
the action of Hecke operator $T_n$ on each point $(E_{(b_i,c_i)},(0,0))$ leads to a sequence of elliptic curves. Using the inverse
transformation  formulae these curves give the point sets on the affine curve $f(x,y)=0$.  Further, we have the corresponding places
of the function field $\mathbb{K}$.

For $P$ a place of $\mathbb{K}=\mathbb{F}_q(X_1(\ell))$, which is not equal to any of the cusps of $X_1(\ell)$, and $D$ a divisor
of $\mathbb{K}$ consists of such places, we define

\begin{definition} Let $P$, $D$ as above and  $T_n$  a Hecke operator. $T_n(P)$ is defined to be the divisor of $\mathbb{K}$,
corresponding to the point set
$\sum_{i=1}^d T_n(E_{(b_i,c_i)},(0,0))$,
which is effective of degree $\Psi(n)d$, where $\Psi(n)=n\prod_{p|n}(1+\frac{1}{p})$.
 Decompose $D $ as $\sum_{i=1}^m a_i P_i$, where $P_i$ are places, then $T_n(D)$ is defined to be $\sum_{i=1}^m a_iT_n(P_i)$.
\end{definition}

Let $D_0$ be a fixed degree one place of $\mathbb{K}$, which will be served as an origin. Every element of $\textrm{Cl}^0(\mathbb{K})$
can be represented by $D-gD_0$, where $g$ is the genus of $\mathbb{K}$, and $D$ is an effective divisor of degree $g$.
Addition in $\textrm{Cl}^0(\mathbb{K})$ means that, given effective divisors $A$ and $B$ of degree $g$, find an effective
divisor $D$ of degree $g$, such that $D-gD_0$ is linearly equivalent to $A-gD_0+B-gD_0$. The complexity of doing this
operation can be found in \cite{Hess},
\begin{prop}Notation is as above. There exists a constant $\omega\in[2,4]$ such that, the divisor $D$ can be computed
in $\textrm{O}(g^{\omega})$ multiplications in the field $\mathbb{F}_q$, i.e. $\textrm{O}(g^{\omega}\log^{1+\epsilon}q)$ bit operations.
\end{prop}
\begin{remark}We have not yet seen the precise value of $\omega$.
However, using  Khuri-Makdisi's algorithm the complexity of a single group operation is known, i.e. $\omega=2.376$,
when fast algorithms for the linear algebra are used.
\end{remark}

The following definition in \cite{Hess} is very useful in our algorithm.
\begin{definition}Let $A$ be a divisor with $\deg(A)\ge1$. The divisor $\tilde{D}$ is called maximally reduced
along $A$ if $\tilde{D}\ge 0$ and $\textrm{dim}(\tilde{D}-rA)=0 $ holds for all $r\ge1$. Let now $D$ be any divisor and $\tilde{D}$ is a divisor
maximally reduced along $A$ such that $D$ is linearly equivalent to $\tilde{D}+rA$ for some $r\in\mathbb{Z}$. Then $\tilde{D}$ is called a
reduction of $D$ along $A$.
\end{definition}

If $\deg(A)=1$, then the reduction divisor $\tilde{D}$ is effective and unique.

\subsection{Searching an $\ell$-torsion point}
One of the main steps in computing the Ramanujan subspace $V_\ell\mod p$ is to find an $\ell$-torsion point in $J_1(\ell)$. In order to do that we have
to work with some large enough extension field $\mathbb{F}_q/\mathbb{F}_p$, such that  $J_1(\ell)(\mathbb{F}_q)$ contains $\ell$-torsion points.
A direct way to get such a point is, pick a random point $Q_0$ in
$J_1(\ell)(\mathbb{F}_q)$ and then compute $Q_1:=N_\ell Q_0$, where $N_\ell$ is the prime-to-$\ell$ part of $\#J_1(\ell)(\mathbb{F}_q)$. If $Q_1$ is
nonzero and $\ell$-torsion, then we succeed. Otherwise, try $Q_2:=\ell Q_1$, check again, after several steps, we can obtain a nonzero $\ell$-torsion
point. As $\#J_1(\ell)(\mathbb{F}_q)$ is bounded above by $q^g$, using fast exponentiation, the running time of getting an $\ell$-torsion point is
about $\log (q^g)\cdot \textrm{O}( g^{\omega})=\textrm{O}(g^{1+\omega}\log q)$ multiplications in the field $\mathbb{F}_q$
or $\textrm{O}(g^{1+\omega}\log^{2+\epsilon} q )$ bit operations. The computation is costly, due to the huge factor $N_\ell$,
which is nearly $q^g$. However, we can accelerate the calculation by introducing some tricks.

For each newform $f$ in $S_2(\Gamma_1(\ell))$, there is an associated abelian variety $A_f$, and $\prod_f A_f$ is an isogeny
decomposition of $J_1(\ell)$, where $f$ runs through a set of representatives for the Galois conjugacy classes of newforms
in $S_2(\Gamma_1(\ell))$. Let $f_\ell$ be the newform, which is congruent  to $\Delta(q)$ modulo $\ell$. Then the Ramanujan
subspace $V_\ell$ lands inside $A_{f_\ell}$.
Let $A':=\prod_f'A_f$ be the product of the abelian varieties, except $A_{f_\ell}$. The minimal polynomial of the Hecke operator $T_2$
acting on $A'$ is denoted by $P_2[X]\in\mathbb{Z}[X]$. Now for a random point $Q_0$ in $J_1(\ell)(\mathbb{F}_q)$,
$Q_1:=P_2(T_2)(Q_0)$ is a point in $A_{f_\ell}(\mathbb{F}_q)$. Similarly, let $n_\ell$ be the prime-to-$\ell$ part of $\#A_{f_\ell}(\mathbb{F}_q)$,
by computing $n_\ell Q_1$, we can obtain an $\ell$-torsion point of $J_1(\ell)(\mathbb{F}_q)$. Since the calculation of $P_2(T_2)(Q_0)$ is easy and
$n_\ell$ is smaller than $N_\ell$, which is bounded above by $q$ to the dimension of $A_{f_\ell}$, we can accelerate the algorithm, especially when
the dimension of $A_{f_\ell}$ is small compared to $g$. However, it seems hard to get a theoretical bound of the dimension of $A_{f_\ell}$. We remark
here several small examples

\begin{center}
\renewcommand\arraystretch{1.5}
\begin{tabular}{ccccccccccccc}
 \hline
    Level $\ell$    & 13 &  17 & 19 & 29 & 31 &  37  & 41  &43 &47 &53  &59  &61  \\
 \hline
   $\dim J_1(\ell)$ & 2  &  5  & 7  & 22 & 26 &  40  & 51  &57 &70 &92  &117 &126\\
 \hline
   $\dim A_{f_\ell}$& 2  &  4  & 6  & 12 & 4 &   18  & 6   &36 &66 &48  &112  &8\\
 \hline
 $\dim J_H(\ell)$   &    &     &    &    & 6 &       &11   &   &  &    &    &26\\
 \hline
\end{tabular}
\end{center}

Another method, suggested by Maarten Derickx, works perfectly when the level $\ell$ satisfying $\ell\equiv 1\mod10$. More precisely, let $\chi$ the
Dirichlet character associated to the newform $f_\ell$, then we have for any prime $p\not=\ell$, $\chi(p)\cdot p\equiv p^{11}\mod \ell$, hence
$\chi(p^{\frac{\ell-1}{10}})\equiv p^{\ell-1}\equiv 1\mod\ell$.
Which means that $f_\ell$ is invariant under the action of diamond operators in the
cyclic subgroup $H=\{p^{\frac{\ell-1}{10}}:p\not=\ell~ \textrm{prime}\}$ of $G=(\mathbb{Z}/\ell\mathbb{Z})^\times$. The Ramanujan subspace $V_\ell$
lands inside the Jacobian variety $J_H(\ell)$, where $J_H(\ell)$ is isogenous to the Jacobian of the modular curve $X_H(\ell)$ associated to the
subgroup of $\SL_2(\mathbb{Z})$ of matrices $[a,b;c,d]$ with $c$ divisible by $\ell$ and $a$ in $H$ modulo $\ell$. We can apply the algorithm to
the modular curve $X_H(\ell)$ and the Jacobian variety $J_H(\ell)$ instead of $X_1(\ell)$ and $J_1(\ell)$, respectively. This useful observation
enables us to carry out the calculation of the level $\ell=31$ case.

We now explain how to compute  $\#J_1(\ell)(\mathbb{F}_q)$.

\begin{lemma}(Manin, Shokurov, Merel, Cremona). For $\ell$ a prime and $p\not\in\{5,\ell\}$ another prime, the Zeta function of
$X_1(\ell)_{\mathbb{F}_p}$ can be computed in deterministic polynomial time in $\ell$ and $p$.
\end{lemma}

\begin{remark}Given this Zeta function we can easily compute $\#J_1(\ell)(\mathbb{F}_q)$, for $q$ is a power of $p$, see \cite{Couveignes}.
Similarly, we can compute $\#A_f(\mathbb{F}_q)$ in deterministic polynomial time in $\ell$ and $p$, by applying the algorithm to newforms in $[f]$,
where $[f]=\{f^\sigma~|~\sigma\in\GQ \}$.
\end{remark}

\subsection{Distinguishing the rational cusps}
In our algorithm, we will compute the action of Hecke operator $T_n$ on points of $J_1(\ell)$. Especially, we need to know the action of $T_n$ on the
$\mathbb{Q}$-rational cusps of $X_1(\ell)$. Notice that, $T_n,n\in\mathbb{Z}^+$ is defined over $\mathbb{Q}$,
so it maps $\mathbb{Q}$-rational point to $\mathbb{Q}$-rational point. As far as we know, there is no an easy way to know the action directly from
only the places $O_i,i\in \{1,\ldots,\frac{\ell-1}{2}\}$. But, a coset representatives of the $\mathbb{Q}$-rational cusps of $X_1(\ell)$ is clear,
i.e. \{$\frac{1}{1},\frac{1}{2},\ldots,\frac{1}{(\ell-1)/2}$\}. And for prime $n\ne \ell$, we have
$T_n(\frac{1}{m})=\frac{1}{m}+n\frac{1}{\overline{nm}}$, where $\overline{nm}$ is the class of integer $nm$ in
 $(\mathbb{Z}/\ell\mathbb{Z})^*/\{\pm1 \}$. So knowing the 1-1 correspondence between
$\{\frac{1}{i},1\le i\le \frac{\ell-1}{2}\}$ and $\{O_i,1\le i\le \frac{\ell-1}{2} \}$ leads to knowing the action of $T_n$ on $O_i$. In general,
it's not easy to know the correspondence \cite{HEON}. Our strategy is to reduce the problem to finite field, as follows. Choose a prime $p$,
such that $\#J_1(\ell)(\mathbb{F}_p)$ has a small factor $d$, let $g$ be the genus of the function field $\mathbb{F}_p(X_1(\ell))$.
Fix a cusp $O_i$ served as origin. As discussed above, let $D-gO_i$ be a degree zero random divisor of order $d$, in general $D$ doesn't contain cusps,
if it does, try a new one. Now, assume $O_i$ corresponds to $\frac{1}{m}$ for some $m\in\{1,\ldots,\frac{\ell-1}{2}\}$ and $O_j$ corresponds to
$\frac{1}{\overline{nm}}$ for some $j\in\{1,\dots,\frac{\ell-1}{2}\}$. Then, for prime $n$, compute $D_n:=T_n(D-gO_i)$ by the assumption as follows
$$D_n=T_n(D)-gT_n(O_i)=T_n(D)-g(O_i+nO_j).$$
If $D_n$ is not of order $d$, then the assumption is wrong. Replace $O_i$ by another $\frac{1}{m}$, or $\frac{1}{\overline{nm}}$ by another $O_j$,
try again. The complete correspondence can be detected after several tries\footnote{We can use the diamond operators instead, which is faster,
suggested by Maarten Derickx}. In fact, the correspondence is known up to cyclic permutation, but it is enough for our algorithm.
There is one more thing should be noticed, the degree 0 divisors of the form $O_i-O_j,1\le i,j\le \frac{\ell-1}{2}$ generate a
subgroup of $J_1(\ell)(\mathbb{F}_p)$, which is called cuspidal subgroup, so the chosen factor $d$ should not be a multiple of the order of the
cuspidal subgroup.

As the example in Section 2, the correspondence between $\{O_1,\ldots,O_9\}$ and $\{\frac{1}{1},\ldots,\frac{1}{9}\}$ is

\begin{displaymath}
\begin{array}{ccccccccc}
\frac{1}{1} &\frac{1}{5} &\frac{1}{6} &\frac{1}{7} &\frac{1}{4} &\frac{1}{3} &\frac{1}{2} &\frac{1}{8} &\frac{1}{9}\\

\updownarrow & \updownarrow &\updownarrow &\updownarrow &\updownarrow &\updownarrow &\updownarrow &
\updownarrow &\updownarrow    \\

O_1 & O_2 & O_3 & O_4 & O_5 &O_6 &O_7&O_8&O_9.
\end{array}
\end{displaymath}

After the above preparation, we can now estimate the complexity of computing $T_n(Q)$, where $Q$ is a degree zero $\mathbb{F}_q$-divisor of the
form $\sum_{i=1}^m a_i P_i-gO$, $O$ is fixed to be the cusp $O_1$ and $n$ is a prime not equal to $\ell$.

As mentioned above, we first compute the point set of the degree $d_i$ place $P_i$ and pick one of them forms a point $(E_{(b_i,c_i)},(0,0))$ on
the modular curve, which is defined over $\mathbb{F}_{q^{d_i}}$. Denote the factorization of $\Phi_n(X,j(E_{(b_i,c_i)}))$ over $\mathbb{F}_{q^{d_i}}[X]$
as $\prod_{j=1}^h F_j(X)$, with $F_j(X)$ irreducible of degree $f_j$. Now for each root of $F_j(X)=0$, there is an isogeny of degree $n$ defined over
$\mathbb{F}_{q^{d_if_j}}$. Computing this isogeny takes
$$\textrm{O}(n^{1+\epsilon} \cdot (\log q^{d_if_j})^{1+\epsilon} )=\textrm{O}((nd_if_j\log q)^{1+\epsilon})$$
bit operations.

Notice that, isogenous curves corresponding to the roots of $F_j(X)$ forms a $\textrm{Gal}(\mathbb{F}_{q^{d_if_j}}/\mathbb{F}_{q^{d_i}})$-conjugate
set. So it suffices to compute any one of them. So the complexity of computing $T_n((E_{(b_i,c_i)},(0,0)))$ is about
$$C_i:=\sum_{j=1}^h \textrm{O}((nd_if_j\log q)^{1+\epsilon}).$$
Since $\sum_{j=1}^h f_j=n+1$, $C_i$ is bounded above by $\textrm{O}((nd_in\log q)^{1+\epsilon})$.

The complexity of computing $T_n(Q)$ is $C:=\sum_{i=1}^m C_i$. Since $\sum_{i=1}^m a_i d_i=g$, $C$ is bounded above by
$\textrm{O}((n^2g\log q)^{1+\epsilon} )$ bit operations.

Notice that after the action of $T_n$, the divisor $T_n(Q)$ becomes complicated. We would like to simplify it before going into further calculation.
In general, the reduction of $T_n(Q)$ along the degree one divisor $O$ comes down to performing at most $n$ additions in the Jacobian,
with a complexity of $\textrm{O}(ng^\omega\log^{1+\epsilon} q)$ bit operations.

\section{Computing the coefficients of modular forms}

After some modification, the simplified algorithm proposed in \cite{Couveignes}, can be used to compute the Ramanujan subspace
$V_\ell\mod p$ efficiently. We can even give a complexity analysis of the algorithm. Roughly speaking, the algorithm works as follows: Step one,
pick some random points in $J_1(\ell)(\mathbb{F}_q)$. Step two, construct $\ell$-torsions points by multiplying these points by some suitable
factors of $\#J_1(\ell)(\mathbb{F}_q)$. Step three, project the $\ell$-torsions points into the space $V_\ell\mod p$ by the Hecke operators.
Step four, reconstruct $V_\ell/\mathbb{Q}$ from sufficiently many $V_\ell\mod p$, by the Chinese Remainder Theorem.

\subsection{Computing the Ramanujan subspace modulo $p$}

The characteristic polynomial of the Frobenius endomorphism $\Frob_p$ acting on the Ramanujan subspace $V_\ell\mod p$ is
$X^2-\tau(p)X+p^{11}\mod \ell$ and the  field of definition of each point of $V_\ell\mod p$ is an extension of $\mathbb{F}_{p}$
with degree $\le d_p$, where
$$d_p=\min_t\{t:X^t\equiv1\mod (X^2-\tau(p)X+p^{11},\ell),t\ge1\}.$$
For convenience, we define
\begin{definition}\label{s-goodprime}Notations as above, a prime $p$ is called $s$-good if $d_p<\ell$.
\end{definition}

The following proposition shows that nearly half of the primes are $s$-good.

\begin{prop}Let $d_p$ defined as above, then we have
$$\lim_{x\to\infty}\frac{|\{p:p~ \textrm{prime},d_p<\ell,p<x\}|}{|\{p:p~\textrm{prime},p<x\}|}\ge\frac{\ell^3-\ell^2-2\ell+2}{2(\ell^3-\ell)}.$$
\end{prop}
\begin{proof}Let $\rho_\ell:\GQ\to\GL_2(\mathbb{F}_\ell)$ be the mod-$\ell$ representation associated to the discriminant modular form $\Delta(q)$.
Denote $K_\ell$ the fixed field of $\ker\rho_\ell$. Then $\rho_\ell$ factors through
$\rho_\ell:\textrm{Gal}(K_\ell/\mathbb{Q})\to\GL_2(\mathbb{F}_\ell)$, which is unramified outside $\ell$.
Now, for a prime $p$ not equal to $\ell$, we have $\rho_\ell(\Frob_p)\in \GL_2(\mathbb{F}_\ell)$ and $d_p$ is the order of the matrix
$\rho_\ell(\Frob_p)$. By the Chebotarev Density Theorem, we have for any conjugacy class $C$ of $G:=\textrm{Gal}(K_\ell/\mathbb{Q})$
the set $\{p:p~\textrm{a prime},p\not=\ell,\Frob_p\in C\}$ has density $|C|/|G|$. Define  $\mathcal{C}=\cup_{\ord(C)<\ell} C$,
where $\ord(C)$ represents the order of any element in $C$, then we have
$$\lim_{x\to\infty}\frac{|\{p:p~ \textrm{prime},d_p<\ell,p<x\}|}{|\{p:p~\textrm{prime},p<x\}|}=\frac{|\mathcal{C}|}{|G|}.$$
Hence we would like to know the image of the representation $\rho_\ell$.

If $\ell$ is an exceptional prime for $\Delta(q)$, i.e. $\ell\in\{2,3,5,7,23,691\}$, then the representation $\rho_\ell$ is reducible
or $\textrm{Im}(\rho_\ell)$ in $\GL_2(\mathbb{F}_{23})$ is dihedral. We can check that $\frac{|\mathcal{C}|}{|G|}\ge\frac{1}{2}$.

If $\ell$ is not  exceptional, we have $\textrm{Im}(\rho_\ell)=\{g\in \GL_2(\mathbb{F}_\ell): \det(g)\in(\mathbb{F}_\ell^\times)^{11}\}$.
The representatives of conjugacy classes in $\GL_2(\mathbb{F}_\ell)$ are as follow

$c_1(x):=\left[\begin{smallmatrix} x & 0 \\ 0 & x \end{smallmatrix}\right], x\in\mathbb{F}_\ell^\times,$

$c_2(x):=\left[\begin{smallmatrix} x & 1 \\ 0 & x \end{smallmatrix}\right], x\in\mathbb{F}_\ell^\times,$

$c_3(x,y):=\left[\begin{smallmatrix} x & 0 \\ 0 & y \end{smallmatrix}\right], x\not=y\in\mathbb{F}_\ell^\times,c_3(x,y)=c_3(y,x),$

$c_4(z):=\left[\begin{smallmatrix} x & Dy \\ y & x \end{smallmatrix}\right],
z=x+\sqrt{D}y\in\mathbb{F}_{\ell^2} \setminus\mathbb{F}_{\ell},c_4(z)=c_4(\bar z)$, where $\bar z:=x-\sqrt{D}y.$

Here $c_3(x,y)=c_3(y,x)$ means that the conjugacy classes of these two elements agree. Let $C_1(x)$ be the conjugacy class with
representative $c_1(x)$, then we have $|C_1(x)|=1$, $\ord(C_1(x))|(\ell-1)$ and there are $N_1=\ell-1$ such classes. Similarly,
we have $|C_2(x)|=\ell^2-1$, $\ord(C_2(x))|\ell(\ell-1)$, $N_2=\ell-1$; $|C_3(x,y)|=\ell(\ell+1)$, $\ord(C_3(x,y))|(\ell-1)$,
$N_3=\frac{1}{2}(\ell-1)(\ell-2)$; $|C_4(z)|=\ell(\ell-1)$, $\ord(C_4(z))|(\ell^2-1)$, $N_4=\frac{1}{2}\ell(\ell-1)$,
we can see $|C_1(x)|N_1+|C_2(x)|N_2+|C_3(x,y)|N_4+|C_4(z)|N_4=|\GL_2(\mathbb{F}_\ell)|=(\ell-1)^2\ell(\ell+1)$.

The subgroup $\textrm{Im}(\rho_\ell)$ consists of those conjugacy classes whose representative matrices have determinant
in $(\mathbb{F}_\ell^\times)^{11}$. For example $\textrm{Im}(\rho_\ell)$ contains conjuagcy classes with representatives $c_1(x)$
satisfying $x^2\in(\mathbb{F}_\ell^\times)^{11}$. Denote the number of such classes by $M_1$ and set $L:=|(\mathbb{F}_\ell^\times)^{11}|$.
Then we have $M_1=L$. Similarly $M_2=L$, $M_3=\frac{1}{2}L(\ell-2)$ and $M_4=\frac{1}{2}\ell L$. So we have
$$\frac{|\mathcal{C}|}{|G|}\ge\frac{M_1+M_3\cdot
\ell(\ell+1)}{M_1+M_2\cdot(\ell^2-1)+M_3\cdot\ell(\ell+1)+M_4\cdot\ell(\ell-1)}=\frac{\ell^3-\ell^2-2\ell+2}{2(\ell^3-\ell)}.$$
\end{proof}

Now let $p$ be an  $s$-good prime and $\mathbb{F}_q:=\mathbb{F}_{p^{d_p}}$, then $V_\ell\mod p$ is a subgroup of $J_1(\ell)(\mathbb{F}_q)[\ell]$.

For every integer $n\ge 2$, the characteristic polynomial of $T_n$ acting on $S_2(\Gamma_1(\ell))$ is a degree $g$ monic polynomial belonging
to $\mathbb{Z}[X]$. We denote it by $A_n(X)$, which can be factored as
$$A_n(X)\equiv B_n(X)(X-\tau(n))^{e_n}~\mod \ell,$$
with $B_n(X)$ monic and $B_n(\tau(n))\not=0 \in \mathbb{F}_\ell$. The exponent $e_n$ is $\ge 1$ due to the theorem of congruence of modular
forms (Thm2.5.7 of \cite{Edixhoven}).
We call $\pi_n:J_1(\ell)(\mathbb{F}_q)[\ell]\to J_1(\ell)(\mathbb{F}_q)[\ell] $ the projection map. Which maps an $\ell$-torsion
point $Q\in J_1(\ell)(\mathbb{F}_q)[\ell] $ to an $\ell$-torsion point $B_n(T_n)(Q)$ of $J_1(\ell)(\mathbb{F}_q)[\ell]$,
and maps bijectively $V_\ell\mod p$ onto itself. Assume $E:=\pi_n(Q)\not=0$, define the exponent $d_n$  as the nonnegative integer satisfying
$$(T_n-\tau(n))^{{d_n}}(E)\not=0~\textrm{and} ~(T_n-\tau(n))^{{d_n}+1}(E)=0,$$
then $d_n$ is in $[0,e_n)$, since $(T_n-\tau(n))^{e_n}(E)=0$.
Let $\tilde{\pi}_n$ be the composition map of $\pi_n$ and $(T_n-\tau(n))^{d_n}$ and $\pi_{\mathcal{S}}:=\prod_{n\in\mathcal{S}}\tilde{\pi}_n$,
where $\mathcal{S}$ is an optimal set defined in Definition \ref{definition:optimalset}. Then we have $\pi_{\mathcal{S}}(Q)\in V_\ell\mod p$ .

So, the complexity of finding a nonzero point in $V_\ell\mod p$  can be determined as following.

As described in Section 3, it takes $\textrm{O}(g^{1+\omega}\log^{2+\epsilon} q )=\textrm{O}(\ell^{4+2\omega+\epsilon}\log^{2+\epsilon}p)$
bit operations to get an $\ell$-torsion point of $J_1(\ell)(\mathbb{F}_q)$. Denote the $\ell$-torsion point as $Q_0=D-gO$, where $D$ is
an effective divisor of degree $g$. For $n\in \mathcal{S}$, the map $\tilde{\pi}_n$ can be written as $T_n^d+a_{d-1}T_n^{d-1}+\ldots+a_1T_n+a_0$,
where $a_i\in\mathbb{F}_\ell$ and $d<g$. The divisor $\tilde{\pi}_n(Q_0)$ can be computed recursively, i.e. compute and simplify(as mentioned
in Section 3.2)
$Q_{i+1}:=T_n(Q_i)$ for $i=0,\ldots,d-1$. The complexity of each step is
$\textrm{O}((n^2g\log q)^{1+\epsilon})+\textrm{O}(ng^\omega\log^{1+\epsilon} q)$,
which is $\textrm{O}(\ell^{2+2\omega+\epsilon}\log^{1+\epsilon}p)$, since $n\in \textrm{O}(\ell)$, $g\in\textrm{O}(\ell^2)$
and $\omega\in[2,4]$. As $d$ is in $\textrm{O}(g)$, computing and simplifying $Q_0,\ldots,Q_d$ takes
$d\cdot\textrm{O}(\ell^{2+2\omega+\epsilon}\log^{1+\epsilon}p)=\textrm{O}(\ell^{4+2\omega+\epsilon}\log^{1+\epsilon}p)$.
Given $Q_i$, since $a_i\in\mathbb{F}_\ell$ the complexity of computing $Q_d+a_{d-1}Q_{d-1}+\ldots+a_1Q_1+a_0Q_0$
is bounded above by $d\ell\cdot\textrm{O}( g^{\omega}\log^{1+\epsilon} q )=\textrm{O}(\ell^{4+2\omega+\epsilon}\log^{1+\epsilon}p)$.
In summary, the complexity of computing $\tilde{\pi}_n(Q_0)$ is still in $\textrm{O}(\ell^{4+2\omega+\epsilon}\log^{1+\epsilon}p)$.
So the complexity of computing $\pi_{\mathcal{S}}(Q_0)=\prod_{n\in \mathcal{S}} \tilde{\pi}_n(Q_0)$ is bounded above by
$\textrm{O}(\ell^{5+2\omega+\epsilon}\log^{1+\epsilon}p)$, since $|\mathcal{S}|<\ell$. So it takes
\begin{equation}\label{complexityofVlmodp}
\textrm{O}(\ell^{4+2\omega+\epsilon}\log^{2+\epsilon}p)+\textrm{O}(\ell^{5+2\omega+\epsilon}\log^{1+\epsilon}p)
=\textrm{O}(\ell^{4+2\omega+\epsilon}\log^{1+\epsilon}p\cdot(\ell+\log p))
\end{equation}
bit operations to get a point in $V_\ell \mod p$.

For two nonzero random points $Q_1,Q_2\in J_1(\ell)(\mathbb{F}_q)[\ell]$, the probability that $\pi_S(Q_1)$ and $\pi_S(Q_2)$
are linearly independent is close to 1-$\frac{1}{\ell}$. So after several attempts we will get a base of $V_\ell\mod p$.
The complexity of checking linearly independence of two elements in $V_\ell\mod p$ is
$\textrm{O}(\ell g^\omega\log^{1+\epsilon} q)=\textrm{O}(\ell^{2+2\omega+\epsilon}\log^{1+\epsilon}p)$.
Hence the complexity of getting a base of $V_\ell \mod p$ is the same as given in (\ref{complexityofVlmodp}).

The Frobenius endomorphism $\Frob_p$ may help us to get a base of $V_\ell\mod p$ with lower cost in some cases.
Namely, if the characteristic polynomial $X^2-\tau(p)X+p^{11}\mod \ell$ is irreducible, then for any nonzero
point $Q\in V_\ell\mod p$, $\Frob_p(Q)$ and $Q$ are linearly independent. So
$V_\ell \mod p=\mathbb{F}_\ell Q+\mathbb{F}_\ell(\Frob_p(Q))$. If $X^2-\tau(p)X+p^{11}\mod\ell$ has two different roots
in $\mathbb{F}_\ell$, then for any nonzero point $Q\in V_\ell\mod p$ the probability that $Q$ and $\Frob_p(Q)$ are linearly independent
is $1-\frac{1}{\ell}$.

\begin{remark}(1)In practice, the optimal set $\mathcal{S}$ contains only small primes, for example, $\mathcal{S}=\{2\}$
for $\ell\in\{13,17,19\}$. The algorithm takes the main effort to get an $\ell$-torsion point.

(2)If without Proposition \ref{propositon:linearlyinl}, the complexity comes up to
\begin{equation}
\textrm{O}(\ell^{4+2\omega+\epsilon}\log^{2+\epsilon}p)+\textrm{O}(\ell^{7+2\omega+\epsilon}\log^{1+\epsilon}p)
=\textrm{O}(\ell^{4+2\omega+\epsilon}\log^{1+\epsilon}p\cdot(\ell^3+\log p)).
\end{equation}
\end{remark}

\subsection{Computing the Ramanujan subspace}
Fix a $\mathbb{Q}$-rational cusp $O$ of $X_1(\ell)$, which will be served as the origin of the Jacobi map.
For every point $x\in V_\ell$, let $D$ be the reduction of $x$ along $O$, i.e. $D=x+dO$. Then $D$ is an effective divisor of degree $d$, which can
be decomposed as $D=Q_1+\ldots+Q_d$, here $d$ is called the stability of $x$, denoted by $\theta(x)$. Choose a rational function
$\psi(x)\in\mathbb{Q}(X_1(\ell))$, which has no pole except at $O$, and define a function $\iota:V_\ell\to \overline{\mathbb{Q}}$
as $\iota(x)=\psi(Q_1)+\ldots+\psi(Q_d)\in \overline{\mathbb{Q}}$. From the uniqueness of $D$, we have $\iota(\sigma(x))=\sigma(\iota(x))$
for any $\sigma\in\textrm{Gal}(\overline{\mathbb{Q}}/\mathbb{Q})$, which is very important. Let $\wp$ be a prime ideal of $\mathbb{Q}(x)$
over $p$, we hope that the uniqueness property still holds in the reduction world, which means the following: let
$\tilde{x},\tilde{O},\tilde{D}$ and $\tilde{Q}_i$ be the reductions of $x,O,D$ and $Q_i$ modulo $\wp$ respectively.
The reduction of $\tilde{x}$ along $\tilde{O}$ is denoted by $\tilde{E}=\tilde{x}+\theta(\tilde{x})\tilde{O}$,
where $\theta(\tilde{x})\le \theta(x)$. If $\theta(\tilde{x})=\theta(x)$, then by the uniqueness property,
we have $\tilde{E}=\tilde{D}=\tilde{Q}_1+\ldots+\tilde{Q}_d$. As we don't know the value $\theta(x)$ in advance,
an algorithm to determine when $\theta(\tilde{x})=\theta(x),~\forall x\in V_\ell$ is needed. In \cite{Bruin},
Bruin gave such an algorithm and proved that, for at least half of the primes smaller than $\ell^\textrm{O(1)}$,
the following holds: $\theta(\tilde{x})=\theta(x),~ \forall x \in V_\ell$ , such primes are called $\mathfrak{m}$-good primes.
In practice, since $\theta(x)$ is less or equal to the genus of $X_1(\ell)$, if $\theta(\tilde{x})$
is equal to the genus for every $\tilde{x}\in V_\ell\mod p$, then $\theta(\tilde{x})=\theta(x)$
automatically for all $x\in V_\ell$. We remark here that our computation suggests that most of the primes are $\mathfrak{m}$-good.

A prime $p$ is called good if it simultaneously satisfies: $\mathfrak{m}$-good and $s$-good. It is reasonable to assume that
there exists an absolute constant $c$ such that the density of good primes is bigger than $c$. Here we make no attempt to prove this.
Now for any good prime $p$
$$P(X):=\prod_{x\in V_\ell\setminus\textrm{O}}(X-\iota(x))$$
is a polynomial in $\mathbb{Q}[X]$ of degree $\ell^2-1$, whose reduction modulo $p$ is exactly the polynomial
$$\tilde{P}(X):=\prod_{\tilde{x}\in V_\ell~\textrm{mod}~p\setminus\textrm{O}}(X-\tilde{\iota}(\tilde{x})),$$
where $\tilde{\iota}$ is the reduction map of $\iota$.

Given $V_\ell \mod p=\mathbb{F}_\ell e_1+\mathbb{F}_\ell e_2$, the computation of  $\tilde{P}(X)$ comes down to performing $\ell^2$
additions in the Jacobian, with a complexity
$\ell^2\cdot\textrm{O}( \ell^{1+2\omega}\log^{1+\epsilon} p)=\textrm{O}( \ell^{3+2\omega}\log^{1+\epsilon} p)$.

Since the heights of coefficients of $P(X)$ are expected to be in $\textrm{O}(\ell^\delta)$, for some absolute constant $\delta$.
Using the fact, there exists a constant $c$ such that
$$\prod_{p\le L,\textrm{prime}}p>c\cdot\exp(L).$$
To recover $P(X)$ from $\tilde{P}(X)$'s, it suffices to take the upper bound $L$ of good primes to be $\textrm{O}(\ell^\delta)$. So the
complexity of computing $P(X)$ will be
$$\sum_{p \le L,\textrm{prime}}\textrm{O}(\ell^{4+2\omega+\epsilon}\log^{1+\epsilon}p\cdot(\ell+\log p))=\ell^{5+2\omega+\delta+\epsilon}.$$
In practice, it would be better to choose the function $\psi(x)\in\mathbb{Q}(X_1(\ell))$, such that the degree is equal to the gonality of
the curve $X_1(\ell)$. For $\ell\le 40$, $\psi(x)$ have been computed by Derickx and Hoeij, see \cite{hoeij} and \cite{Derickx}.

\subsection{Finding the Frobenius endomorphism }

The Galois representation associated to $\Delta(q)$ is denoted by
$\rho_\ell:\textrm{Gal}(\overline{\mathbb{Q}}/\mathbb{Q})\to \textrm{GL}_2(\mathbb{F}_\ell)$.
Let $K_\ell:=\overline{\mathbb{Q}}^{\ker{\rho_\ell}}$ be the field cut out by the representation, then $K_\ell$ is the splitting field of the
polynomial $P(X)$. For any prime $p\not=\ell$, the trace of Frobenius $\textrm{Tr}(\rho_\ell(\textrm{Frob}_p))$ is equal to $\tau(p)\mod\ell$.
So we would like to identify the conjugacy class of the Galois group $\textrm{Gal}(K_\ell/\mathbb{Q})$, where the Frobenius $\textrm{Frob}_p$
lands inside. The algorithm described in \cite{Dokchitser} can be used perfectly to do the computation.

Let $(a_i)_{1\le i\le\ell^2-1}$ be the roots of $P(X)$ in $K_\ell$ and $h(X)$ some polynomial in $\mathbb{Q}[X]$.
Then for each conjugacy class $C\subset \textrm{Gal}(K_\ell/\mathbb{Q})$,
\begin{equation}\label{traceformula}
\textrm{Frob}_p\in C\Leftrightarrow \Gamma_C(\textrm{Tr}_{\frac{\mathbb{F}_p[x]}{P(x)}/\mathbb{F}_p}(h(x)x^p))\equiv0\mod p,
\end{equation}
where the polynomial $\Gamma_C(X)$ is given by
\begin{equation}\label{GammaCX}
\Gamma_C(X)=\prod_{\sigma\in C}\left(X-\sum_{i=1}^{\ell^2-1}h(a_i)\sigma(a_i) \right).
\end{equation}
Our strategy for computing $\Gamma_C(X)$ is using  Hensel lifting.

\begin{prop}Let $p$ be a prime such that the extension degree $d_p$ is in $\textrm{O}(1)$. Given the polynomial $P(X)$ as in Section 4.2
and the Ramanujan subspace $V_\ell\mod p$. Then for any conjugacy class $C\subset \textrm{Gal}(K_\ell/\mathbb{Q})$,
the polynomial $\Gamma_C(X)\in\mathbb{Q}[X]$ can be computed in $\textrm{O}(\ell^{6+\delta+\epsilon})$ bit operations.
\end{prop}
\begin{proof}
Set $\mathbb{F}_{q}:=\mathbb{F}_{p^{d_p}}$. We can lift each root $\tilde{\iota}(x)\in \mathbb{F}_q$ of $\tilde{P}(X)\in \mathbb{F}_p[X]$ to
the root $\iota(x)\in \mathbb{Q}_q$ of  $P(X)\in\mathbb{Q}_p[X]$ by Hensel's lemma, where $\mathbb{Q}_q$ is the unramified extension
of $\mathbb{Q}_p$ with extension degree $d_p$. Since the heights of coefficients of $\Gamma_C(X)$ are bounded above by
$N:=|C|(1+\deg h)\ell^\delta$, to recover $\Gamma_C(X)$, it suffices to lift each $\tilde{\iota}(x)$ to $\iota(x)$ with precision $N$.
The complexity of a single lifting is about $\textrm{O}((|C|(d_pN))^{1+\epsilon})$ bit operations, see \cite{Avanzi-Cohen}.
The length of the conjugacy class $C$ is in  $\textrm{O}(\ell^2)$, the extension degree $d_p$ is in $\textrm{O}(1)$ and $h(X)$
can be chosen to be a polynomial with small degree (e.g. $\deg h(X)=2$ ) showed in \cite{Dokchitser}, so the complexity of lifting a
single root is in $\textrm{O}(\ell^{4+\delta+\epsilon})$. There are $\ell^2-1$ roots need to be lifted, so the complexity of computing
$\Gamma_C(X)$ is $\textrm{O}(\ell^{6+\delta+\epsilon})$.
\end{proof}
\begin{remark}The Galois group $\textrm{Gal}(K_\ell/\mathbb{Q})$ is a subgroup of $\textrm{GL}_2(\mathbb{F}_\ell)$, and the
action of $\textrm{Gal}(K_\ell/\mathbb{Q})$ on the roots $a_i,1\le i\le \ell^2-1$ can be computed from the action of
$\textrm{GL}_2(\mathbb{F}_\ell)$ on $V_\ell \mod p$. In practice, we would like to choose the prime $p$ with large size and small
extension degree $d_p$ at the same time.
\end{remark}

\subsection{Complexity analysis}

Now we can prove Theorem \ref{theorem:complexity} and Corollary \ref{cor:complexity}. As shown in above, the complexity of
computing $P(X)$ is $\textrm{O}(\ell^{5+2\omega+\delta+\epsilon})$. There are $\ell^2-1$ conjugacy classes in
$\textrm{GL}_2(\mathbb{F}_\ell)$. So the complexity of computing $\Gamma_C(X)$ for all the conjugacy classes
$C\subset \textrm{GL}_2(\mathbb{F}_\ell)$ is $\textrm{O}(\ell^{8+\delta+\epsilon})$.

Let $p$ be a prime not equal to $\ell$. Denote the polynomial $P(X)$ as
$$P(X)=X^{\ell^2-1}+c_{\ell^2-2}X^{\ell^2-2}+\ldots+c_0.$$
The trace in (\ref{traceformula}) can be interpreted as a trace of a matrix
$$\textrm{Tr}_{\frac{\mathbb{F}_p[x]}{P(x)}/\mathbb{F}_p}(x^d)=\textrm{Tr}\left(
                                                                            \begin{array}{cccc}
                                                                              0 &  &  & -c_0 \\
                                                                              1 &  &  & -c_1 \\
                                                                              & \ddots &  & \vdots \\
                                                                               &   &   1&-c_{\ell^2-1}  \\
                                                                            \end{array}
                                                                          \right)^d \mod p.
$$
Using Coppersmith-Winograd algorithm, the complexity of multiplying two $\ell^2\times\ell^2$ matrices over $\mathbb{F}_p$ is
in $\textrm{O}(\ell^{4.752}\log^{1+\epsilon}p)$ bit operations. So $t:=\textrm{Tr}_{\frac{\mathbb{F}_p[x]}{P(x)}/\mathbb{F}_p}(h(x)x^p)$
can be computed in $\textrm{O}(\ell^{4.752}\log^{2+\epsilon}p)$. Given $t$, the complexity of checking whether $\Gamma_C(t)\mod p$ is equal
to zero for all the conjugacy classes $C$ is bounded above by $\textrm{O}(\ell^4\log^{1+\epsilon}p)$.

So, for any prime $p\not=\ell$, the complexity of computing $\tau(p)\mod\ell$ consists of the complexity of computing $P(X)$, $\Gamma_C(X)$,
$\textrm{Tr}(h(x)x^p)$ and $\Gamma_C(\textrm{Tr}(h(x)x^p))$ for all conjugacy classes $C\subset\GL_2(\mathbb{F}_\ell)$, which sums up to
\begin{equation}
\textrm{O}(\ell^{5+2\omega+\delta+\epsilon})+\textrm{O}(\ell^{4.752}\log^{2+\epsilon}p)+\textrm{O}(\ell^4\log^{1+\epsilon}p).
\end{equation}

For prime $p$, we have $\tau(p)\in \textrm{O}(p^6)$. Therefore it suffices to recover $\tau(p)$ from $\tau(p)\mod \ell$ with $\ell\le L$,
where $L$ is in $\textrm{O}(\log p)$. So the total complexity of computing $\tau(p)$ is about
\begin{equation}
\sum_{\ell<L,\textrm{prime}}\textrm{O}(\ell^{5+2\omega+\delta+\epsilon})+\textrm{O}(\ell^{4.752}\log^{2+\epsilon}p)+
\textrm{O}(\ell^4\log^{1+\epsilon}p)
=\textrm{O}(\log^{6+2\omega+\delta+\epsilon}p).
\end{equation}

\begin{remark}
If the constant $\omega$ reaches $2.376$, the complexity of the algorithm is $\textrm{O}(\log^ {10.752+\delta+\epsilon}p )$.
Moreover, if $\delta$ is bounded above by $3$, the complexity is $\textrm{O}(\log^{13.752+\epsilon}p )$.
\end{remark}

\section{Implementation and results}
The algorithm has been implemented in MAGMA. One big advantage of the algorithm is that, it is rather straightforward to implement,
where the major work is dealing with the action of Hecke operators on divisors of the function field.
The following computation was done on a personal computer AMD FX(tm)-6200 Six-Core Processor 3.8GHz.

Let $Q_\ell(x)$ be the polynomial corresponding to the projective representation, defined as
$$Q_\ell(X):=\prod_{L\in\mathbb{P}(V_\ell)}(X-\sum_{\alpha\in L\setminus\textrm{O}}{\iota}(\alpha)),$$
which can be used to check whether $\tau(p)\equiv 0 \mod\ell$. More precisely, we have the following lemma, see \cite{Bosman}
\begin{lemma}Let $Q_\ell(X)$ be the polynomial defined as above and $p\nmid\textrm{Disc}(Q_\ell(X))$ a prime. Then
$\tau(p)\equiv 0 \mod\ell$ if and only if $Q_\ell(X)\mod p$ has an irreducible factor of degree 2 over $\mathbb{F}_p$.
\end{lemma}

\begin{example}$\ell=13$.

To recover $Q_{13}(X)$, it suffices to take a good prime set as
$$\{ 19, 23, 29, 43, 53, 61, 67, 71, 79, 83, 89, 109, 127, 149,
157, 163, 179, 193, 211, 223, 229, 233, 239, 241 \},$$
with a total of 24 primes, whose product is a 52 digits number. We have
\begin{displaymath}
\begin{split}
2535853\cdot Q_{13}(X)=&2535853X^{14} + 760835865X^{13} + 96570870461X^{12} + 7083218145770X^{11} +\\
    &341554192651282X^{10} + 11596551892957577X^9 + 288394789072144586X^8 +\\
    &5369247990154339694X^7 + 75509842125272520446X^6 +\\
    &800346109631330635243X^5 + 6303044886777591079517X^4 +\\
    &35793920471135235999031X^3 + 138667955645963961606844X^2 +\\
    &328650624808255716476451X + 361128579432826593902125.
\end{split}
\end{displaymath}
The computation took several minutes.

The product of the primes needed to recover $P_{13}(X)$ is about 130 digits and the computation took nearly one hour.

In order to recover the polynomials $\Gamma_C(X)$ for $C\subset\textrm{GL}_2(\mathbb{F}_{13})$, we first choose a good prime
$p=34939$ with $d_p=2$, and then compute the roots of $P_{13}(X) \mod p$ by computing $V_{\ell}\mod p$. Notice that all of the roots
are in $\mathbb{F}_{p^2}$. Using Hensel lemma, we lift each root to the $p$-adic field $\mathbb{Q}_{p^2}$ with a
precision nearly 5000 digits and then reconstruct $\Gamma_X(X)$ by the formula (\ref{GammaCX}). The computation took several hours.
\end{example}

We summarize the results in the following table, where the matrices in the last column are the representatives of the conjugacy
classes where the Frobenius endomorphism $\Frob_p$ land inside and  prime $p$ is set to be $10^{1000}+1357$.

\begin{center}
\renewcommand\arraystretch{1.5}
\begin{tabular}{ccccc}

\hline
level &
\begin{tabular}{cc}

\multicolumn{2}{c}{$Q_\ell(X)$} \\
good primes &  time\\
\end{tabular} &

\begin{tabular}{cc}

\multicolumn{2}{c}{$P_\ell(X)$} \\

good primes & time \\

\end{tabular} &
\begin{tabular}{c}

$\Gamma_C(X)$ \\

time \\

\end{tabular} &
$\Frob_p$  \\
 \hline
13&
\begin{tabular}{cc}

52 digits & few minutes\\

\end{tabular} &

\begin{tabular}{cc}

130 digits & few minutes \\

\end{tabular}&
few hours &
$\left[\begin{smallmatrix} 10 & 0 \\ 0 & 7 \end{smallmatrix}\right]$\\
 \hline
17&
\begin{tabular}{cc}

467 digits & few hours \\

\end{tabular} &

\begin{tabular}{cc}

740 digits & few hours \\

\end{tabular}&
one day &
$\left[\begin{smallmatrix} 15 & 1\\ 0 & 15 \end{smallmatrix}\right]$\\
 \hline
19&
\begin{tabular}{cc}

832 digits & few days \\

\end{tabular} &

\begin{tabular}{cc}

1681 digits & few days \\

\end{tabular}&
few days &
$\left[\begin{smallmatrix} 17 & 1 \\ 0 & 17 \end{smallmatrix}\right]$\\
\hline
\end{tabular}

\end{center}

Using a plane model for $X_H(31)$, we also finished the level 31 case. It took several days to recover the polynomial $Q_{31}(X)$.
The coefficients of $Q_{31}(X)$ are very large, where the biggest one reaches 2426 digits. Similar to the $\ell\in\{13,17,19\}$ cases,
$Q_{31}(X)$ can be reduced to a polynomial with small coefficients. One of the reduced polynomials is
\begin{displaymath}
\begin{split}
f_{31}= &X^{32} - 4X^{31} - 155X^{28} + 713X^{27} - 2480X^{26} + 9300X^{25} - 5921X^{24}\\
&+ 24707X^{23} + 127410X^{22} - 646195X^{21} + 747906X^{20} - 7527575X^{19} +\\
&4369791X^{18} - 28954961X^{17} - 40645681X^{16} + 66421685X^{15} - 448568729X^{14}\\
&+ 751001257X^{13} - 1820871490X^{12} + 2531110165X^{11} - 4120267319X^{10} +\\
&4554764528X^9 - 5462615927X^8 + 4607500922X^7 - 4062352344X^6 + 2380573824X^5\\
&- 1492309000X^4 + 521018178X^3 - 201167463X^2 + 20505628X - 1261963.
\end{split}
\end{displaymath}
Set $K=\mathbb{Q}[X]/(f_{31})$ and $\mathcal{O}_K$ the maximal order of $K$, then we can check that the discriminant of $\mathcal{O}_K$
is equal to $-31^{41}$, and the Galois group of the polynomial $f_{31}$ is isomorphic to $\textrm{PGL}_2(\mathbb{F}_{31})$.
These facts prove $f_{31}$ is the polynomial corresponding to the mod-31 projective representation associated to $\Delta(q)$.
An easy calculation shows that the first few primes satisfying  Serre's criteria as well as
$\tau(p)\equiv0\mod 11\cdot 13\cdot 17\cdot 19\cdot 31$ are
$$982149821766199295999,3748991773540147199999,$$
$$3825907566871689215999,3903375187595059199999.$$
So we proved Theorem \ref{theorem:main2}.

From the table above, we have $\tau(10^{1000}+1357)\equiv 15\mod19$. So the missing sign in the table of \cite{Edixhoven} is found.
Moreover, we have
$$\tau(10^{1000}+1357)\equiv\pm 18\mod 31.$$

The Magma code of our algorithm can be downloaded from the web at the address

\begin{center}\url{http://faculty.math.tsinghua.edu.cn/~lsyin/publication.htm}\end{center}

\section{acknowledgments}
Our interest in computing coefficients of modular forms is motivated by the wonderful courses given by Bas Edixhoven and Jean-Marc Couveignes at
Tsinghua University. Many thanks to them for their encouragement. The first author wishes  to thank Jean-Marc Couveignes for his
continuous assistance and many helpful suggestions. Thanks to Ye Tian for his helpful comments. Many thanks to Maarten Derickx
for fruitful discussions, helpful comments and suggestions. Thanks to Mark van Hoeij for providing us a plane model for $X_H(31)$ and
helping us reduce the polynomial $Q_{31}(X)$, these make the computation of level 31  a reality.

\end{document}